\RequirePackage{fix-cm}
\documentclass[smallextended,envcountsect,nospthms]{svjour3}

\usepackage[T1]{fontenc}
\usepackage[utf8]{inputenc}
\usepackage{mathptmx}
\usepackage{amsmath,amssymb,amsthm,mathtools}
\usepackage{graphicx}
\usepackage{tikz}
\usetikzlibrary{arrows.meta,calc,positioning}
\usepackage{bm}
\usepackage{xcolor}
\usepackage{url}
\usepackage{hyperref}

\smartqed
\journalname{Journal of Evolution Equations}

\hypersetup{
  colorlinks=true,
  linkcolor=blue!55!black,
  citecolor=blue!55!black,
  urlcolor=blue!55!black,
  pdftitle={Scattering and low-speed critical elements for the three-dimensional focusing energy-critical nonlinear Schrodinger equation},
  pdfauthor={Pang-Hung Chung and Dan Han}
}
\allowdisplaybreaks
\numberwithin{equation}{section}

\newtheoremstyle{spbreak}%
  {3pt}{3pt}%
  {\itshape}{}%
  {\bfseries}{.}%
  {\newline}%
  {}%
\newtheoremstyle{spdefbreak}%
  {3pt}{3pt}%
  {}{}%
  {\bfseries}{.}%
  {\newline}%
  {}%
\theoremstyle{spbreak}
\newtheorem{theorem}{Theorem}[section]
\newtheorem{proposition}[theorem]{Proposition}
\newtheorem{lemma}[theorem]{Lemma}

\theoremstyle{spdefbreak}

\theoremstyle{spdefbreak}

\makeatletter
\@ifundefined{acknowledgements}{%
  \newenvironment{acknowledgements}{\section*{Acknowledgements}}{}%
}{}
\@ifundefined{at}{\newcommand{\at}{\par\smallskip}}{}
\makeatother

\newcommand{\R}{\mathbb R}
\newcommand{\C}{\mathbb C}
\newcommand{\eps}{\varepsilon}

\newcommand{\Vc}{\mathcal V}
\newcommand{\Err}{\operatorname{Err}}

\newcommand{\Nd}{\dot N^1}
\newcommand{\Sd}{\dot S^1}

\begin{document}

\title{Scattering and Low-Speed Critical Elements for the Three-Dimensional Focusing Energy-Critical Nonlinear Schr\"odinger Equation}
\titlerunning{Scattering and low-speed critical elements}

\author{Pang-Hung Chung \and Dan Han\thanks{Corresponding author.}}
\authorrunning{P.-H. Chung and D. Han}

\institute{%
P.-H. Chung \at
Department of Applied Mathematics, Guangdong University of Education,
Guangzhou 510640, P. R. China\\
\email{penghongzhong@yahoo.com}
\and
D. Han \at
Department of Mathematics, University of Louisville,
Louisville, KY 40245, USA\\
\email{dan.han@louisville.edu}
}

\date{}

\maketitle

\begin{abstract}
The below-threshold scattering problem is considered for the three-dimensional focusing energy-critical nonlinear Schr\"odinger equation. A main non-radial obstruction is the possible drift of the concentration center of a soliton-like critical element, which prevents a fixed-center localized virial estimate from closing directly. It is shown that such a compact critical solution must vanish if it has an $L_t^\infty L_x^q$ bound and its center drifts sufficiently slowly. The result covers the pure-energy, endpoint $L^4$, and finite-mass regimes and gives a corresponding conditional scattering criterion.
\end{abstract}

%\section*{Contents}
%\noindent
%\begin{tabular}{@{}ll@{}}
%1. & Introduction\\
%2. & Notation, thresholds, and coercivity\\
%3. & Critical compactness, tails, and drift\\
%4. & Local mass growth and the standard $L^p$ breach\\
%5. & Fixed-center localized virial\\
%6. & Proofs of the main results\\
%   & Acknowledgements\\
%   & References
%\end{tabular}
\medskip

\section{Introduction}

The three-dimensional energy-critical nonlinear Schr\"odinger equation is a basic model in critical dispersive theory. In this paper we consider the focusing equation
\begin{equation}\label{eq:NLS}
  i\partial_tu+\Delta u+|u|^4u=0,
  \qquad (t,x)\in I\times\R^3,
\end{equation}
where $I\subset\R$ is a time interval and $u:I\times\R^3\to\C$. The equation is invariant under the scaling
\[
  u(t,x)\mapsto u_\lambda(t,x):=\lambda^{1/2}u(\lambda^2t,\lambda x),
  \qquad \lambda>0,
\]
and
\[
  \|\nabla u_\lambda(0)\|_{L_x^2}=\|\nabla u(0)\|_{L_x^2}.
\]
Thus the natural critical space is
\[
  \dot H^1(\R^3):=\{f\in L^6(\R^3):\nabla f\in L^2(\R^3)\},
  \qquad \|f\|_{\dot H^1}:=\|\nabla f\|_{L^2}.
\]
Throughout the paper we write
\[
  C_t^0\dot H_x^1(J\times\R^3)=C(J;\dot H^1(\R^3)),
  \qquad U_0:=U(0),
\]
set $B(z,R):=\{x\in\R^3:|x-z|<R\}$, and use $X\lesssim_\alpha Y$ to mean $X\le C(\alpha)Y$ for a constant $C(\alpha)>0$ depending only on $\alpha$. We write $X\gtrsim_\alpha Y$ when $Y\lesssim_\alpha X$, and $X\simeq_\alpha Y$ when both $X\lesssim_\alpha Y$ and $Y\lesssim_\alpha X$ hold. We also write $F(T)=O(G(T))$ if there exists a constant $C>0$ such that $|F(T)|\le C G(T)$ for all sufficiently large $T$, and $F(T)=o(G(T))$ if $F(T)/G(T)\to0$ as $T\to\infty$, whenever $G(T)>0$. Finally, $\mathcal K\Subset X$ means that $\mathcal K$ is relatively compact in the normed space $X$. For a time interval $J$, set $L_{t,x}^p(J\times\R^3):=L^p(J\times\R^3)$ and define the critical scattering norm
\[
  S_J(u):=\|u\|_{L_{t,x}^{10}(J\times\R^3)}.
\]
For the equation \eqref{eq:NLS}, local well-posedness, stability, and small-data scattering in the energy space $\dot H^1$ follow from Strichartz estimates and a contraction argument. The foundational literature includes the works of Ginibre--Velo, Kato, Cazenave--Weissler, Keel--Tao, and Cazenave's monograph \cite{GinibreVelo1979I,GinibreVelo1979II,GinibreVelo1985,Kato1987,CazenaveWeissler1990,KeelTao1998,Cazenave2003}. For strong solutions, the energy
\[
  E(f):=\frac12\int_{\R^3}|\nabla f|^2\,dx-
       \frac16\int_{\R^3}|f|^6\,dx
\]
is conserved. We also use the Pohozaev functional
\[
  K(f):=\int_{\R^3}|\nabla f|^2\,dx-
       \int_{\R^3}|f|^6\,dx.
\]
The sharp threshold for the focusing problem is determined by the Aubin--Talenti optimizers for the Sobolev inequality \cite{Aubin1976,Talenti1976}. We fix
\[
  W(x)=\left(1+\frac{|x|^2}{3}\right)^{-1/2}.
\]
Then $-\Delta W=W^5$, $\|\nabla W\|_{L^2}^2=\|W\|_{L^6}^6$, $K(W)=0$, and
\[
  E(W)=\frac13\|\nabla W\|_{L^2}^2.
\]
The sharp Sobolev inequality and the Weinstein variational mechanism \cite{Weinstein1983} yield the below-threshold potential well
\[
  E(u_0)<E(W),\qquad \|\nabla u_0\|_{L^2}<\|\nabla W\|_{L^2}.
\]

In this region the solution remains on the stable branch below the ground state. If the solution is nonzero, then $K(u(t))$ has a positive lower bound throughout its interval of existence. This positive virial gap is the driving force behind the rigidity argument below.

The general non-radial below-threshold unconditional scattering problem considered here may be summarized as follows: under the above below-threshold assumptions, must every global strong solution satisfy $S_\R(u)<\infty$ and hence scatter? For the defocusing three-dimensional energy-critical equation, global well-posedness and scattering were established by Colliander--Keel--Staffilani--Takaoka--Tao \cite{CKSTT2008}; related radial and high-dimensional defocusing results are due to Bourgain, Tao, Tao--Visan, and Visan \cite{Bourgain1999,Tao2005,TaoVisan2005,Visan2007}. In the focusing case, the concentration-compactness/rigidity method of Kenig--Merle proves the radial below-threshold scattering/blow-up dichotomy in dimensions three, four, and five \cite{KenigMerle2006}. Their contradiction mechanism relies on profile decomposition and a minimal nonscattering critical element, and these tools are closely related to the work of Bahouri--G\'erard and Keraani \cite{BahouriGerard1999,Keraani2001}. The refined dynamics near the threshold was classified by Duyckaerts--Merle \cite{DuyckaertsMerle2009}. The virial blow-up mechanism in the finite-variance or radial setting goes back to Glassey and Ogawa--Tsutsumi \cite{Glassey1977,OgawaTsutsumi1991}, and sharp scattering thresholds in the mass-supercritical, energy-subcritical model give important analogies for the energy-critical problem \cite{HolmerRoudenko2008,DuyckaertsHolmerRoudenko2008}.

The radial hypothesis plays a decisive role in the Kenig--Merle rigidity step: the spatial center of a radial critical element is fixed at the origin, and a localized virial truncation only needs to cover a fixed ball. In the non-radial case, concentration compactness usually gives precompactness only modulo scaling and translation. In other words, if $U$ is a minimal nonscattering object, then there may only exist a scale $N(t)$ and a center $x(t)$ such that
\[
  \left\{N(t)^{-1/2}U\left(t,x(t)+\frac{y}{N(t)}\right):t\ge0\right\}
  \Subset\dot H^1(\R^3).
\]
Denote by $R_T>0$ a fixed-center localized virial radius on the time window $[0,T]$. This radius must grow with $\sup_{0\le t\le T}|x(t)-x(0)|$. At the same time, the upper bound for the virial action also grows with $R_T$. Under pure $\dot H^1$ control, the local $L^2$ mass can only be estimated by the rough Sobolev bound; accordingly, the action upper bound is only of order $O(R_T^2)$. This is exactly the source of the subdiffusive exponent $T^{1/2}$ for the drift.

We now fix two quantities that measure drift and local mass growth. The channel considered in this paper is the soliton-like critical channel, namely there exist constants $N_-,N_+>0$ such that the scale function satisfies
\[
  0<N_-\le N(t)\le N_+<\infty.
\]
Fixing the observation center $z=x(0)$, define the center-drift radius by
\[
  D_x(T):=\sup_{0\le t\le T}|x(t)-x(0)|.
\]
For $2\le q\le6$, define the local mass growth exponent
\[
  \theta(q):=3\left(\frac12-\frac1q\right).
\]
Thus $\theta(6)=1$ corresponds to pure energy control, $\theta(4)=3/4$ corresponds to the endpoint $L^4$ control, and $\theta(2)=0$ corresponds to finite mass.

More concretely, the fixed-center virial argument closes around three quantities. Let $A>0$ be a tail radius and set $R_T:=D_x(T)+A$. Denote by $\mathcal V_{R_T}(t)$ the fixed-center localized virial action and by $\operatorname{Err}_{R_T}(t;z)$ its truncation remainder; their precise expressions are given in Section~\ref{sec:virial}. The derivative has the form
\[
  \frac{d}{dt}\mathcal V_{R_T}(t)
  =8K(u(t))+\operatorname{Err}_{R_T}(t;z),
\]
where the truncation error is controlled by
\begin{align*}
  |\operatorname{Err}_{R_T}(t;z)|
  &\lesssim
  \int_{|x-z|\ge R_T}\bigl(|\nabla u(t,x)|^2+|u(t,x)|^6\bigr)\,dx\\
  &\quad +R_T^{-2}\int_{R_T\le |x-z|\le 2R_T}|u(t,x)|^2\,dx.
\end{align*}
Critical compactness and the triangle inequality ensure that, once $A$ is fixed sufficiently large, this error can be made small uniformly on the whole time window $0\le t\le T$. This step uses only the fact that $R_T$ covers the region swept out by the center drift. On the other hand, the action itself satisfies
\[
  |\mathcal V_{R_T}(t)|
  \lesssim R_T\|u(t)\|_{L^2(B(z,2R_T))}.
\]
Hence the non-radial difficulty is compressed precisely into the question of how the local $L^2$ mass grows with $R_T$. The parameter $q$ records this growth rate: $q=6$ is the pure-energy layer, $q=4$ is the endpoint layer, and $q=2$ is the finite-mass layer. The contradiction is then obtained not by choosing one fixed radius once and for all, but along a time sequence $T_n\to\infty$ satisfying
\[
  \frac{(D_x(T_n)+1)^{1+\theta(q)}}{T_n}\to0.
\]
This sequential closing is a technical point that allows us to treat low-speed drift without imposing any monotonicity assumption on the drift.

There has been major progress on non-radial focusing problems in higher dimensions. Killip--Visan proved non-radial below-threshold scattering for the focusing energy-critical equation in dimensions five and higher \cite{KillipVisan2010}, and Dodson treated the four-dimensional focusing energy-critical cubic equation \cite{Dodson2019}. These works rely on dimension-dependent long-time Strichartz estimates, double Duhamel arguments, additional integrability, or negative-regularity mechanisms. The three-dimensional quintic energy-critical equation lies at a lower-dimensional endpoint, and the low-frequency behavior of the ground state $W$ obstructs a direct derivation of finite-mass $L^2$ control from pure $\dot H^1$ critical compactness. Thus the general three-dimensional non-radial below-threshold unconditional scattering problem remains a key difficulty. This paper does not claim to solve that full problem; instead, it excludes a family of low-speed soliton-like critical channels for which a fixed-center virial argument can be closed.

In the above soliton-like framework, if $2\le q\le6$ and $u\in L_t^\infty L_x^q$, then the local mass growth satisfies
\[
  \|u(t)\|_{L^2(B(x(t),R))}
  \lesssim R^{\theta(q)}\|u\|_{L_t^\infty L_x^q}.
\]
Consequently, the fixed-center virial closure takes the form
\[
  T\lesssim \bigl(D_x(T)+1\bigr)^{1+\theta(q)}.
\]
The main novelty of this paper is to close the below-threshold coercivity, smallness of the critical compactness tails, and local mass growth in a unified way, obtaining a rigidity criterion with the parameter $q$. This includes, in a single framework, the pure-energy regime, the Killip--Visan type $L^p$ breach \cite{KillipVisan2013}, the endpoint $L^4$ regime, and the finite-mass regime.

We emphasize that the low-speed condition here constrains only the growth order of the drift radius. It does not assume any regularity of $x(t)$, existence of a velocity, or monotonicity. On each time window $[0,T]$ the proof fixes a tail radius $A_0>0$, chooses $R_T:=D_x(T)+A_0$, and uses the same truncation function on that window. Thus the moving center may oscillate back and forth. It is enough that there exists a sequence $T_n\to\infty$ such that
\[
  (D_x(T_n)+1)^{1+\theta(q)}=o(T_n)
\]
in order to contradict the linear virial growth. By contrast, if the center has almost linear drift, for instance $D_x(T)\simeq T$, then the fixed-center virial action upper bound and the linear lower bound live on the same scale, and the present method no longer gives a contradiction by itself. This explains why the full unconditional scattering problem must combine additional dynamical information in order to rule out fast channels or to show that they cannot arise as minimal counterexamples.

Technically, the rigidity of a non-radial soliton-like critical element is reduced to the proportion
\[
  \frac{\text{upper bound for the virial action}}{\text{linear virial lower bound}}
  \lesssim
  \frac{(D_x(T)+A_0)^{1+\theta(q)}}{T}.
\]
The below-threshold hypothesis produces only the linear lower bound in the denominator; critical compactness ensures that the truncation error does not destroy this lower bound; and the $L_t^\infty L_x^q$ input controls only the local mass growth in the numerator. This separation allows the proof below to substitute different values of $q$ within the same framework, and it makes clear that the endpoint obstruction in three dimensions is concentrated in the finite-mass or negative-regularity issue.

With these preparations in place, the rigidity statement can be formulated as follows.

\begin{theorem}[Low-speed critical elements]\label{thm:rigidity-general}
Let $U\in C_t^0\dot H_x^1([0,\infty)\times\R^3)$ be a global strong solution to \eqref{eq:NLS}. Assume
\[
  E(U_0)<E(W),
  \qquad \|\nabla U_0\|_{L^2}<\|\nabla W\|_{L^2}.
\]
Assume that there exist $N(t)$ and $x(t)$ such that
\[
  \left\{N(t)^{-1/2}U\left(t,x(t)+\frac{y}{N(t)}\right):t\ge0\right\}
  \Subset\dot H^1(\R^3),
\]
and
\[
  0<N_-\le N(t)\le N_+<\infty,
  \qquad t\ge0.
\]
Assume moreover that, for some $q\in[2,6]$,
\[
  M_q(U):=\sup_{t\ge0}\|U(t)\|_{L^q(\R^3)}<\infty,
\]
where, when $q=6$, this condition follows automatically from critical compactness, boundedness of the scale, and Sobolev embedding. If the center drift satisfies the sequential virial admissibility condition
\begin{equation}\label{eq:rigidity-drift-general}
  \liminf_{T\to\infty}
  \frac{(D_x(T)+1)^{1+\theta(q)}}{T}=0,
\end{equation}
then $U\equiv0$.
\end{theorem}

Theorem \ref{thm:rigidity-general} gives four closed regimes. First, taking $q=6$, Sobolev embedding gives $\theta(6)=1$, and the closure condition reduces to
\[
  \frac{(D_x(T)+1)^2}{T}\to0
  \quad\text{along some sequence }T\to\infty.
\]
In particular, it excludes $D_x(T)=o(T^{1/2})$. Second, if $U$ is a standard critical element satisfying the no-waste Duhamel structure, then the Killip--Visan type $L^p$ breach gives $U\in L_t^\infty L_x^p$, $4<p<6$. In this case $\theta(p)>3/4$ but may be taken arbitrarily close to $3/4$, and hence one excludes
\[
  D_x(T)=o(T^\beta),
  \qquad \beta<\frac47.
\]
Third, if in addition the endpoint bound $U\in L_t^\infty L_x^4([0,\infty)\times\R^3)$ holds, then $q=4$ and the theorem excludes
\begin{equation}\label{eq:L4-endpoint-sequence-intro}
  \liminf_{T\to\infty}\frac{(D_x(T)+1)^{7/4}}{T}=0,
\end{equation}
in particular $D_x(T)=o(T^{4/7})$. Fourth, if there is an additional finite-mass input, or a finite-mass input obtained from negative regularity by interpolation,
\[
  U\in L_t^\infty L_x^2([0,\infty)\times\R^3),
\]
then $q=2$, and the complete sublinear drift condition
\[
  D_x(T)=o(T)
\]
is enough to exclude the critical element. This shows that pushing the endpoint $T^{4/7}$ low-speed version to the full sublinear version depends not on the virial computation itself, but on whether one can obtain finite mass or an equivalent negative-regularity control for three-dimensional non-radial compact critical solutions.

If $u$ is a global strong solution on $[0,\infty)$ and there exists $u_+\in\dot H^1(\R^3)$ such that
\[
  \lim_{t\to\infty}\|u(t)-e^{it\Delta}u_+\|_{\dot H^1}=0,
\]
we say that $u$ scatters forward in time. The second result is a conditional scattering criterion. To avoid confusion with the set notation used in Section \ref{sec:notation}, we denote the family of initial data by $\mathfrak A$. The criterion says that, once every counterexample can be reduced to a low-speed soliton-like channel excluded by Theorem \ref{thm:rigidity-general}, scattering follows.

\begin{theorem}[Conditional scattering]\label{thm:scattering-conditional}
Let $\mathfrak A\subset\dot H^1(\R^3)$ be a family of below-threshold initial data, that is, every $u_0\in\mathfrak A$ satisfies
\[
  E(u_0)<E(W),
  \qquad \|\nabla u_0\|_{L^2}<\|\nabla W\|_{L^2}.
\]
Assume that $\mathfrak A$ has the following low-speed soliton-like critical-element reduction property: if there exists $u_0\in\mathfrak A$ such that the corresponding forward global strong solution $u$ satisfies $S_{[0,\infty)}(u)=\infty$, then there exists a nonzero global strong solution $U$ satisfying all assumptions of Theorem \ref{thm:rigidity-general}. Then every forward global strong solution $u$ corresponding to initial data $u_0\in\mathfrak A$ satisfies
\[
  S_{[0,\infty)}(u)<\infty,
\]
and hence scatters forward in time.
\end{theorem}

The remainder of the paper is organized as follows. Section \ref{sec:notation} fixes the energy space, the Strichartz notation, the sharp Sobolev threshold, and the below-threshold coercivity. Section \ref{sec:compactness} converts critical compactness into uniform tail smallness and records the geometric consequences of the drift radius. Section \ref{sec:Lq} proves the local mass growth under an $L^q$ bound and explains how the $L^p$ breach, the endpoint $L^4$ input, and the finite-mass input change the drift exponent. Section \ref{sec:virial} proves the fixed-center localized virial identity, the truncation error estimate, and the action upper bound. Section \ref{sec:main-proofs} proves Theorems \ref{thm:rigidity-general} and \ref{thm:scattering-conditional} in a unified way.

\section{Notation, thresholds, and coercivity}\label{sec:notation}

This section keeps the notation $\dot H^1$, $E$, $K$, $S_J$, $D_x$, and $\theta(q)$ introduced in the introduction, and only adds the function-space norms and variational threshold facts needed later. Unless otherwise specified, all Lebesgue norms are taken over $\R^3$. For a time interval $J$, set
\begin{equation*}
  \|u\|_{L_t^qL_x^r(J\times\R^3)}
  :=\left(\int_J\left(\int_{\R^3}|u(t,x)|^r\,dx\right)^{q/r}\,dt\right)^{1/q},
\end{equation*}
where, when $q=\infty$, the outer norm is the essential supremum over $t\in J$, and, when $r=\infty$, the inner norm is the essential supremum over $x\in\R^3$. A pair $(q,r)$ is called three-dimensional Schr\"odinger admissible if
\begin{equation*}
  2\le q,r\le\infty,
  \qquad (q,r,3)\ne(2,\infty,3),
  \qquad \frac2q+\frac3r=\frac32.
\end{equation*}
For such a pair, $q'$ and $r'$ denote the H\"older conjugates, defined by
\begin{equation*}
  \frac1q+\frac1{q'}=1,
  \qquad
  \frac1r+\frac1{r'}=1,
  \qquad \frac1\infty:=0.
\end{equation*}
Define
\begin{equation*}
  \|u\|_{\Sd(J)}:=\sup_{(q,r)\text{ admissible}}\|\nabla u\|_{L_t^qL_x^r(J\times\R^3)},
\end{equation*}
and the dual nonlinear norm
\begin{equation*}
  \|F\|_{\Nd(J)}
  :=\inf_{(q,r)\text{ admissible}}\|\nabla F\|_{L_t^{q'}L_x^{r'}(J\times\R^3)}.
\end{equation*}
For an arbitrary reference time $t_0\in J$, the Strichartz estimates give
\begin{equation*}
  \left\|\int_{t_0}^{t}e^{i(t-s)\Delta}F(s)\,ds\right\|_{\Sd(J)}
  \lesssim \|F\|_{\Nd(J)}.
\end{equation*}
For the energy-critical nonlinearity $F(u)=|u|^4u$, the chain rule and H\"older's inequality yield the standard estimate
\begin{equation}\label{eq:critical-nonlinearity-estimate}
  \||u|^4u\|_{\Nd(J)}
  \lesssim \|u\|_{L_{t,x}^{10}(J\times\R^3)}^4\|u\|_{\Sd(J)}.
\end{equation}
Therefore, if $\|e^{i(t-t_0)\Delta}u(t_0)\|_{L_{t,x}^{10}(J\times\R^3)}$ is sufficiently small, a strong solution on $J$ is constructed by contraction. We only need the following direct consequence. If on an interval $J$ one has $S_J(u)\le\eta$, with $\eta>0$ sufficiently small, then Duhamel's formula, Strichartz estimates, and \eqref{eq:critical-nonlinearity-estimate} give
\[
\begin{aligned}
  \|u\|_{\Sd(J)}
  &\lesssim \|u(t_0)\|_{\dot H^1}
     +\||u|^4u\|_{\Nd(J)}  \\
  &\lesssim \|u(t_0)\|_{\dot H^1}
     +S_J(u)^4\|u\|_{\Sd(J)}.
\end{aligned}
\]
Choosing $\eta$ so that the implicit constant times $\eta^4$ is less than $1/2$, the last term on the right can be absorbed, and one obtains
\[
  \|u\|_{\Sd(J)}\lesssim \|u(t_0)\|_{\dot H^1}.
\]
If $S_{[0,\infty)}(u)<\infty$, then $[0,\infty)$ can be divided into finitely or countably many $S$-small intervals, and the $S$ norm on tail intervals tends to zero. This fact will be used in the Cauchy argument showing that finite critical norm implies scattering.

To avoid repeatedly writing the gradient and critical nonlinear terms, set
\begin{equation*}
  A(f):=\|\nabla f\|_{L^2}^2,
  \qquad B(f):=\|f\|_{L^6}^6,
  \qquad A_W:=A(W).
\end{equation*}
From the definitions of $E$ and $K$ in the introduction one has
\begin{equation}\label{eq:E-via-A-K}
  E(f)=\frac13A(f)+\frac16K(f),
  \qquad A(f)=3E(f)-\frac12K(f).
\end{equation}
The first identity is particularly useful in proving a positive virial gap for nonzero solutions.

We next record two local $L^2$ estimates that will be used to bound the virial action. The Sobolev inequality gives, for a constant $C_3>0$,
\begin{equation*}
  \|f\|_{L^6}\le C_3\|\nabla f\|_{L^2}.
\end{equation*}
Hence, for any $R>0$ and $z\in\R^3$,
\begin{equation}\label{eq:local-L2-basic}
  \|f\|_{L^2(B(z,R))}
  \le |B(z,R)|^{1/3}\|f\|_{L^6}
  \lesssim R\|\nabla f\|_{L^2}.
\end{equation}
If $f\in L^q(\R^3)$, $2\le q\le6$, then
\begin{equation}\label{eq:local-L2-Lq}
  \|f\|_{L^2(B(z,R))}
  \le |B(z,R)|^{\frac12-\frac1q}\|f\|_{L^q}
  \lesssim R^{\theta(q)}\|f\|_{L^q},
\end{equation}
where $\theta(q)$ is the exponent defined in the introduction.

We now return to the below-threshold structure. Let $C_{\rm opt}>0$ denote the optimal constant in the sharp Sobolev inequality, written as
\begin{equation*}
  \|f\|_{L^6}^2\le C_{\rm opt}\|\nabla f\|_{L^2}^2.
\end{equation*}
and its optimizers are given by the Aubin--Talenti family. Our chosen $W$ satisfies
\begin{equation*}
  -\Delta W=W^5.
\end{equation*}
Multiplying by $W$ and integrating gives
\begin{equation*}
  \|\nabla W\|_{L^2}^2=\|W\|_{L^6}^6.
\end{equation*}
Thus
\begin{equation*}
  K(W)=0,
  \qquad E(W)=\frac13\|\nabla W\|_{L^2}^2.
\end{equation*}

By the sharp Sobolev inequality and the normalization of the ground state,
\begin{equation*}
  B(f)\le A_W\left(\frac{A(f)}{A_W}\right)^3.
\end{equation*}
Let
\begin{equation*}
  y:=\frac{A(f)}{A_W}.
\end{equation*}
Then
\begin{align}
  E(f)&\ge A_W\left(\frac12y-\frac16y^3\right),\label{eq:E-lower-y}\\
  K(f)&\ge A_W(y-y^3).\label{eq:K-lower-y}
\end{align}

The one-dimensional function
\begin{equation*}
  g(y):=\frac12y-\frac16y^3,
  \qquad g'(y)=\frac12(1-y^2),
\end{equation*}
is strictly increasing on $[0,1]$ and satisfies $g(1)=1/3$. Therefore the below-threshold assumption, with $y(t):=A(u(t))/A_W$, gives not only $y(t)<1$, but also a quantitative gap. More precisely, if
\begin{equation*}
  \eta_E:=1-\frac{E(u_0)}{E(W)}>0,
\end{equation*}
then there exist $\delta=\delta(\eta_E,\|\nabla u_0\|_{L^2}/\|\nabla W\|_{L^2})>0$ and $c(\delta)>0$ such that
\begin{equation*}
  \frac{A(u(t))}{A(W)}\le1-\delta,
  \qquad K(u(t))\ge c(\delta)A(u(t)).
\end{equation*}
Below we only need the existence of such a gap, not the explicit form of $c(\delta)$. Applying this one-dimensional potential-well estimate to the flow gives the below-threshold branch preservation.

\begin{lemma}[Below-threshold branch]\label{lem:coercivity-branch}
Let $u$ be an energy solution to \eqref{eq:NLS}, and assume
\begin{equation*}
  E(u_0)<E(W),
  \qquad \|\nabla u_0\|_{L^2}<\|\nabla W\|_{L^2}.
\end{equation*}
Then there exists $\delta_1=\delta_1(u_0)>0$ such that, for all $t$,
\begin{equation*}
  \|\nabla u(t)\|_{L^2}^2\le(1-\delta_1)\|\nabla W\|_{L^2}^2,
\end{equation*}
and
\begin{equation}\label{eq:K-coercive-gradient}
  K(u(t))\ge\delta_1\|\nabla u(t)\|_{L^2}^2.
\end{equation}
\end{lemma}

\begin{proof}
Set
\begin{equation*}
  y(t):=\frac{\|\nabla u(t)\|_{L^2}^2}{\|\nabla W\|_{L^2}^2}.
\end{equation*}
By \eqref{eq:E-lower-y} and conservation of energy,
\begin{equation}\label{eq:energy-barrier}
  E(W)>E(u_0)=E(u(t))
  \ge A_W\left(\frac12y(t)-\frac16y(t)^3\right).
\end{equation}
The function $g(y)$ defined above is strictly increasing on $[0,1]$ and satisfies
\begin{equation*}
  g(1)=\frac13=\frac{E(W)}{A_W}.
\end{equation*}
The initial condition gives $y(0)<1$. By the continuity of $t\mapsto y(t)$ and the energy barrier \eqref{eq:energy-barrier}, the orbit cannot cross $y=1$, and there exists $\delta_1>0$ such that
\begin{equation*}
  y(t)\le1-\delta_1.
\end{equation*}
Using \eqref{eq:K-lower-y}, we obtain
\begin{equation*}
  K(u(t))\ge A_Wy(t)(1-y(t)^2)
  \ge c(\delta_1)A_Wy(t)
  =c(\delta_1)\|\nabla u(t)\|_{L^2}^2.
\end{equation*}
Renaming the constant gives \eqref{eq:K-coercive-gradient}.
\end{proof}

The next step turns gradient coercivity into a strictly positive virial gap; this is used directly in the lower bound for the localized virial derivative.

\begin{lemma}[Positive virial gap]\label{lem:positive-K}
Under the assumptions of Lemma \ref{lem:coercivity-branch}, if $u\not\equiv0$, then there exists
\begin{equation*}
  \kappa_0=\kappa_0(u_0)>0
\end{equation*}
such that
\begin{equation}\label{eq:K-kappa0}
  K(u(t))\ge\kappa_0,
  \qquad t\ge0.
\end{equation}
\end{lemma}

\begin{proof}
By \eqref{eq:K-coercive-gradient},
\begin{equation*}
  K(u(t))\ge\delta_1A(u(t)).
\end{equation*}
On the other hand, by \eqref{eq:E-via-A-K} and \eqref{eq:K-coercive-gradient}, whenever $A(u(t))>0$ at some time,
\begin{equation*}
  E(u_0)=E(u(t))=\frac13A(u(t))+\frac16K(u(t))>0.
\end{equation*}
A nonzero strong solution cannot have $A(u(t))=0$ at any time, since uniqueness would then imply $u\equiv0$. Hence $E(u_0)>0$. Also,
\begin{equation*}
  E(u(t))=\frac12A(u(t))-\frac16B(u(t))\le\frac12A(u(t)),
\end{equation*}
and therefore
\begin{equation*}
  A(u(t))\ge2E(u_0)>0.
\end{equation*}
Taking
\begin{equation*}
  \kappa_0:=2\delta_1E(u_0)
\end{equation*}
proves the claim.
\end{proof}

We conclude this section with a standard consequence of the local theory: a finite critical scattering norm is enough to construct the scattering state.

\begin{proposition}[Finite norm scattering]\label{prop:finite-S-scattering}
Let $u$ be a global strong solution on $[0,\infty)$ and assume
\begin{equation*}
  S_{[0,\infty)}(u)<\infty.
\end{equation*}
Then there exists $u_+\in\dot H^1(\R^3)$ such that
\begin{equation*}
  \lim_{t\to\infty}\|u(t)-e^{it\Delta}u_+\|_{\dot H^1}=0.
\end{equation*}
\end{proposition}

\begin{proof}
By local theory and continuity, $[0,\infty)$ can be partitioned into finitely or countably many intervals $J_k$ on which $\|u\|_{L_{t,x}^{10}(J_k)}$ is sufficiently small. Using \eqref{eq:critical-nonlinearity-estimate} and a stability iteration yields
\begin{equation*}
  \|u\|_{\Sd([0,\infty))}<\infty.
\end{equation*}
For $0<T_1<T_2$, Duhamel's formula and the dual Strichartz estimate give
\begin{align*}
  &\left\|e^{-iT_2\Delta}u(T_2)-e^{-iT_1\Delta}u(T_1)\right\|_{\dot H^1}\\
  &\qquad\le
  \left\|\int_{T_1}^{T_2}e^{-is\Delta}\bigl(|u(s)|^4u(s)\bigr)\,ds\right\|_{\dot H^1}
  \lesssim \||u|^4u\|_{\Nd([T_1,T_2])}\\
  &\qquad\lesssim
  \|u\|_{L_{t,x}^{10}([T_1,T_2]\times\R^3)}^4
  \|u\|_{\Sd([T_1,T_2])}.
\end{align*}
Since $S_{[T_1,\infty)}(u)\to0$ as $T_1\to\infty$, the right-hand side tends to zero. Hence $e^{-it\Delta}u(t)$ is Cauchy in $\dot H^1$, and its limit is $u_+$.
\end{proof}

\section{Critical compactness, tails, and drift}\label{sec:compactness}

Throughout this section, set $z:=x(0)$, and recall that $\Subset$ denotes relative compactness. Set
\begin{equation*}
  v(t,y):=N(t)^{-1/2}u\left(t,x(t)+\frac{y}{N(t)}\right).
\end{equation*}
Critical compactness is written as
\begin{equation}\label{eq:compactness-v}
  \{v(t):t\ge0\}\Subset\dot H^1(\R^3).
\end{equation}
The soliton-like bounded-scale condition is written as
\begin{equation}\label{eq:N-bounded}
  0<N_-\le N(t)\le N_+<\infty,
  \qquad t\ge0.
\end{equation}
Under \eqref{eq:N-bounded}, the physical tails can be controlled uniformly. We first record a tail criterion that depends only on precompactness.

\begin{lemma}[Tail criterion]\label{lem:precompact-tail-criterion}
Let $\mathcal K\Subset\dot H^1(\R^3)$. Then
\begin{equation}\label{eq:precompact-tail-criterion}
  \lim_{A\to\infty}\sup_{f\in\mathcal K}
  \left(
  \int_{|y|\ge A}|\nabla f(y)|^2\,dy
  +\int_{|y|\ge A}|f(y)|^6\,dy
  \right)=0.
\end{equation}
\end{lemma}

\begin{proof}
Given $\eps>0$, precompactness gives some $m\in\mathbb N$ and functions $f_1,\ldots,f_m\in\dot H^1$ such that every $f\in\mathcal K$ satisfies $\|f-f_j\|_{\dot H^1}<\eps$ for some $j$. Sobolev embedding gives $\|f-f_j\|_{L^6}\lesssim\eps$. For the finitely many fixed functions $f_j$, choose $A$ so large that their gradient and $L^6$ tails are all less than $\eps$. The triangle inequality and $(a+b)^6\lesssim a^6+b^6$ then imply \eqref{eq:precompact-tail-criterion}.
\end{proof}

\begin{lemma}[Compactness tails]\label{lem:compact-tail}
If \eqref{eq:compactness-v} and \eqref{eq:N-bounded} hold, then
\begin{equation}\label{eq:tail-limit}
  \lim_{A\to\infty}\sup_{t\ge0}
  \left[
  \int_{|x-x(t)|\ge A}|\nabla u(t,x)|^2\,dx
  +\int_{|x-x(t)|\ge A}|u(t,x)|^6\,dx
  \right]=0.
\end{equation}
\end{lemma}

\begin{proof}
Apply Lemma \ref{lem:precompact-tail-criterion} to the precompact set \eqref{eq:compactness-v}. For any $\eps>0$ there exists $A_0=A_0(\eps)$ such that
\begin{equation*}
  \sup_{t\ge0}\left[
  \int_{|y|\ge A_0}|\nabla v(t,y)|^2\,dy
  +\int_{|y|\ge A_0}|v(t,y)|^6\,dy
  \right]<\eps.
\end{equation*}
The change of variables
\begin{equation*}
  y=N(t)(x-x(t))
\end{equation*}
gives
\begin{align*}
  \int_{|x-x(t)|\ge A}|\nabla u(t,x)|^2\,dx
  &=\int_{|y|\ge AN(t)}|\nabla v(t,y)|^2\,dy,\\
  \int_{|x-x(t)|\ge A}|u(t,x)|^6\,dx
  &=\int_{|y|\ge AN(t)}|v(t,y)|^6\,dy.
\end{align*}
If $A\ge A_0/N_-$, then $AN(t)\ge A_0$, and \eqref{eq:tail-limit} follows.
\end{proof}

For the later truncation-error estimates, define the tail modulus of the compact trajectory by
\begin{equation*}
  \omega(A):=\sup_{t\ge0}
  \left[
  \int_{|x-x(t)|\ge A}|\nabla u(t,x)|^2\,dx
  +\int_{|x-x(t)|\ge A}|u(t,x)|^6\,dx
  \right].
\end{equation*}
Lemma \ref{lem:compact-tail} is equivalent to
\begin{equation*}
  \lim_{A\to\infty}\omega(A)=0.
\end{equation*}
This modulus depends only on the precompact orbit and on $N_-$ and $N_+$, not on the time window $[0,T]$. All later choices of $A_0$ may be understood as follows: first fix a prescribed $\varepsilon_0>0$, then choose $A_0>1$ sufficiently large that
\begin{equation*}
  \omega(A_0)\le\varepsilon_0\kappa_0,
\end{equation*}
and finally let the virial truncation radius grow with $D_x(T)$.

It is important that no global finite mass is needed here. The $L^2$ term in the truncation error appears only on an annulus, in the form
\[
  R^{-2}\int_{R\le |x-z|\le 2R}|u(t,x)|^2\,dx.
\]
Let
\begin{equation*}
  \Omega_R:=\{x\in\R^3:R\le |x-z|\le2R\}.
\end{equation*}
In three dimensions, H\"older's inequality and the critical Sobolev exponent give
\[
\begin{aligned}
  R^{-2}\int_{\Omega_R}|u(t,x)|^2\,dx
  &\le R^{-2}|\Omega_R|^{2/3}
       \left(\int_{\Omega_R}|u(t,x)|^6\,dx\right)^{1/3}  \\
  &\lesssim \left(\int_{\Omega_R}|u(t,x)|^6\,dx\right)^{1/3},
\end{aligned}
\]
Thus, once this annulus lies in the far tail relative to the moving center, the $L^2$ annular term is also controlled by smallness of the $L^6$ tail. This is the key point that allows one to use a localized virial argument at the three-dimensional energy level.

Recall that $z=x(0)$. The drift radius defined in the introduction may then be written as
\begin{equation*}
  D_x(T)=\sup_{0\le t\le T}|x(t)-z|.
\end{equation*}
Given a tail radius $A>0$, define
\begin{equation*}
  R_T:=D_x(T)+A.
\end{equation*}
Then, for $0\le t\le T$,
\begin{equation}\label{eq:geometry-tail}
  |x-z|\ge R_T
  \quad\Longrightarrow\quad
  |x-x(t)|\ge A.
\end{equation}
Indeed,
\begin{equation*}
  |x-x(t)|\ge |x-z|-|x(t)-z|
  \ge D_x(T)+A-D_x(T)=A.
\end{equation*}
Moreover, since $D_x(T)\le R_T$,
\begin{equation}\label{eq:ball-inclusion}
  B(z,2R_T)\subset B(x(t),3R_T),
  \qquad 0\le t\le T.
\end{equation}
Figure~\ref{fig:drift-geometry} records this fixed-center geometry. It is used only to localize the moving compact core inside a single virial window on the interval $[0,T]$.

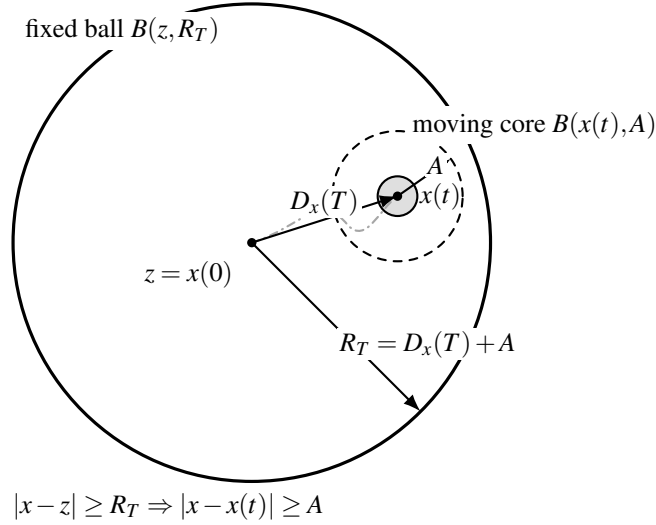
\begin{figure}[t]
\centering
\begin{tikzpicture}[scale=.82, line cap=round, line join=round, >=Latex]
  \coordinate (z) at (0,0);
  \coordinate (xt) at (2.35,.75);
  \def\RT{3.85}
  \def\A{1.05}
  \path[use as bounding box] (-4.35,-4.75) rectangle (4.55,4.15);
  \draw[very thick] (z) circle (\RT);
  \draw[thick,dashed] (xt) circle (\A);
  \fill[gray!25] (xt) circle (.32);
  \draw[thick] (xt) circle (.32);
  \draw[gray!70,dash dot,thick] (0,0) .. controls (.8,.25) and (1.1,.85) .. (1.55,.28)
      .. controls (1.85,-.03) and (2.05,.54) .. (xt);
  \fill (z) circle (.075) node[below left=5pt] {$z=x(0)$};
  \fill (xt) circle (.075) node[right=5pt] {$x(t)$};
  \draw[->,thick] (z) -- node[above,fill=white,inner sep=1pt] {$D_x(T)$} (xt);
  \draw[->,thick] (xt) -- node[above right,fill=white,inner sep=1pt] {$A$} ($(xt)+(35:\A)$);
  \draw[->,thick] (z) -- node[below right,fill=white,inner sep=1pt] {$R_T=D_x(T)+A$} (-45:\RT);
  \node[fill=white,inner sep=2pt,anchor=west] at (-3.75,3.42) {fixed ball $B(z,R_T)$};
  \node[fill=white,inner sep=2pt,anchor=west] at (2.50,1.90) {moving core $B(x(t),A)$};
  \node[fill=white,inner sep=2pt,anchor=west] at (-3.95,-4.25)
    {$|x-z|\ge R_T \Rightarrow |x-x(t)|\ge A$};
\end{tikzpicture}
\caption{Fixed-center geometry for the drift radius. The fixed virial center is $z=x(0)$, the moving concentration center is $x(t)$, and the truncation radius is $R_T=D_x(T)+A$. Thus the exterior region $|x-z|\ge R_T$ lies outside the ball $B(x(t),A)$ for all $0\le t\le T$.}
\label{fig:drift-geometry}
\end{figure}

Define the local mass growth function
\begin{equation*}
  \Phi(R):=\sup_{t\ge0}\|u(t)\|_{L^2(B(x(t),R))}.
\end{equation*}
By \eqref{eq:local-L2-basic} and boundedness of the energy, one always has the rough estimate
\begin{equation*}
  \Phi(R)\lesssim R.
\end{equation*}
If only this estimate is used, the fixed-center virial argument can exclude only the subdiffusive drift $D_x(T)=o(T^{1/2})$. If, in addition, $u\in L_t^\infty L_x^q$, $2\le q\le6$, then \eqref{eq:local-L2-Lq} gives
\begin{equation*}
  \Phi(R)\lesssim R^{\theta(q)}.
\end{equation*}
Consequently, the virial closure condition improves from $D_x(T)^2=o(T)$ to the sequential condition
\begin{equation*}
  \liminf_{T\to\infty}
  \frac{(D_x(T)+1)^{1+\theta(q)}}{T}=0.
\end{equation*}
When $q=2$, this includes the full sublinear drift condition $D_x(T)=o(T)$; when $q=4$, it becomes the endpoint condition \eqref{eq:L4-endpoint-sequence-intro}, in particular including $D_x(T)=o(T^{4/7})$; and when only $q>4$ is available and one lets $q\downarrow4$, it gives $D_x(T)=o(T^\beta)$ for every $\beta<4/7$.

\section{Local mass growth and the standard \texorpdfstring{$L^p$}{Lp} breach}\label{sec:Lq}

This section relates the local mass function $\Phi(R)$ from the previous section to a global $L^q$ bound. We first give the direct H\"older estimate.

\begin{lemma}[Local mass growth]\label{lem:local-mass-Lq}
Let $2\le q\le6$ and use the notation of the theorem. Assume
\begin{equation*}
  M_q(u):=\sup_{t\ge0}\|u(t)\|_{L^q(\R^3)}<\infty.
\end{equation*}
Then there exists a constant $C_q>0$, depending only on $q$, such that, for every $R>0$,
\begin{equation*}
  \Phi(R)
  \le C_q M_q(u)R^{\theta(q)}.
\end{equation*}
\end{lemma}

\begin{proof}
By H\"older's inequality,
\begin{align*}
  \|u(t)\|_{L^2(B(x(t),R))}
  &\le |B(x(t),R)|^{\frac12-\frac1q}\|u(t)\|_{L^q}\\
  &\lesssim R^{3(\frac12-\frac1q)}M_q(u).
\end{align*}
Taking the supremum over $t$ gives the claim.
\end{proof}

This estimate turns the drift problem into an algebraic condition. If $D_x(T)=o(T^\beta)$, then
\[
  \frac{(D_x(T)+1)^{1+\theta(q)}}{T}\to0
  \qquad\text{whenever}\qquad
  \beta(1+\theta(q))<1.
\]
Substituting $\theta(6)=1$, $\theta(4)=3/4$, and $\theta(2)=0$ gives respectively $\beta<1/2$, $\beta<4/7$, and $\beta<1$. At the endpoint $q=4$, the small-$o$ condition $D_x(T)=o(T^{4/7})$ directly implies
\[
  \frac{(D_x(T)+1)^{7/4}}{T}\to0.
\]
Thus all the special drift exponents in the later virial result arise from the same function $\theta(q)$, rather than from different virial identities.

Let $\mathcal S(\R^3)$ denote the Schwartz space, $\mathcal S'(\R^3)$ the space of tempered distributions, and $\mathcal P$ the space of polynomials. For $f\in\mathcal S(\R^3)$, define
\begin{equation*}
  \widehat f(\xi):=(2\pi)^{-3/2}\int_{\R^3}e^{-ix\cdot\xi}f(x)\,dx,
\end{equation*}
and extend the Fourier transform to $\mathcal S'(\R^3)$ by duality. For $\sigma\in\R$, define
\begin{equation*}
  \dot H^\sigma(\R^3)
  :=\left\{f\in\mathcal S'(\R^3)/\mathcal P:
  |\xi|^\sigma\widehat f(\xi)\in L^2(\R^3)\right\},
\end{equation*}
with norm
\begin{equation*}
  \|f\|_{\dot H^\sigma}
  :=\left(\int_{\R^3}|\xi|^{2\sigma}|\widehat f(\xi)|^2\,d\xi\right)^{1/2}.
\end{equation*}
For a time interval $J$, also set
\begin{equation*}
  \|u\|_{L_t^\infty\dot H_x^\sigma(J\times\R^3)}
  :=\mathop{\mathrm{ess\,sup}}_{t\in J}\|u(t)\|_{\dot H^\sigma(\R^3)}.
\end{equation*}

We next record the standard critical-element input used in this paper and include the proof of the reduction used below. The no-waste Duhamel structure means that the critical element satisfies, for every $t\ge0$, the weak limiting representation
\begin{equation*}
  u(t)=-i\lim_{T\to\infty}\int_t^T e^{i(t-s)\Delta}\bigl(|u(s)|^4u(s)\bigr)\,ds
  \quad\text{in }\dot H^1,
\end{equation*}
and, in the standard two-sided critical-element setting used in the double-Duhamel step, also the backward version
\begin{equation*}
  u(t)=i\lim_{T\to-\infty}\int_T^t e^{i(t-s)\Delta}\bigl(|u(s)|^4u(s)\bigr)\,ds.
\end{equation*}
This representation removes the linear radiation component. It is usually obtained from the minimal blow-up reduction and perturbation theory, and it is the starting point for the double Duhamel and $L^p$ breach mechanisms \cite{KillipVisan2010,KillipVisan2013,Dodson2019}.

\begin{proposition}[$L^p$ breach]\label{prop:Lp-breach}
Let $u$ be a global almost periodic critical element for the three-dimensional energy-critical NLS, satisfying
\begin{equation*}
  \{N(t)^{-1/2}u(t,x(t)+\cdot/N(t)):t\ge0\}
  \Subset\dot H^1(\R^3),
\end{equation*}
and
\begin{equation*}
  0<N_-\le N(t)\le N_+<\infty,
  \qquad \sup_{t\ge0}\|\nabla u(t)\|_{L^2}<\infty.
\end{equation*}
If $u$ satisfies the standard no-waste Duhamel structure, then
\begin{equation*}
  \forall p\in(4,6),
  \qquad \|u\|_{L_t^\infty L_x^p([0,\infty)\times\R^3)}<\infty.
\end{equation*}
\end{proposition}

\begin{proof}
The proof is the standard Killip--Visan $L^p$-breach mechanism, written here in the form needed for the present soliton-like channel. Here $C_c^\infty(\R^3)$ denotes the space of smooth compactly supported functions. Fix a radial function $\chi\in C_c^\infty(\R^3)$ satisfying $0\le\chi\le1$, $\chi(\xi)=1$ for $|\xi|\le1$, and $\chi(\xi)=0$ for $|\xi|\ge2$. For each dyadic frequency $M\in2^{\mathbb Z}:=\{2^k:k\in\mathbb Z\}$, define the smooth Littlewood--Paley projections by
\begin{equation*}
  \widehat{P_{\le M}f}(\xi):=\chi(\xi/M)\widehat f(\xi),
  \qquad
  P_M:=P_{\le M}-P_{\le M/2},
  \qquad
  P_{>M}f:=f-P_{\le M}f.
\end{equation*}
All frequency sums below are taken over dyadic $M,K\in2^{\mathbb Z}$; the symbol $K$ inside such sums denotes a frequency and is unrelated to the Pohozaev functional $K(f)$. When no time interval is displayed in a mixed norm in this proof, the time variable ranges over $[0,\infty)$. Put $F(u):=|u|^4u$. The no-waste Duhamel formula gives, weakly in $\dot H^1$,
\begin{align*}
  u(t)&=-i\lim_{T\to\infty}\int_t^T e^{i(t-s)\Delta}P_{\le M}F(u(s))\,ds \\
      &\quad -i\lim_{T\to\infty}\int_t^T e^{i(t-s)\Delta}P_{>M}F(u(s))\,ds.
\end{align*}
The almost periodicity and the soliton-like bound on $N(t)$ imply uniform frequency tightness in the critical norm:
\begin{equation*}
  \lim_{M\downarrow0}\sup_{t\ge0}\|P_{\le M}u(t)\|_{\dot H^1}=0,
  \qquad
  \lim_{M\uparrow\infty}\sup_{t\ge0}\|P_{>M}u(t)\|_{\dot H^1}=0.
\end{equation*}
Equivalently, after normalizing the bounded scale to unit size, the critical orbit is contained in a compact subset of $\dot H^1$ and has neither an ultraviolet tail nor an arbitrarily large infrared tail in the $\dot H^1$ topology.

The frequency-localized reduced-Duhamel estimate of Killip--Visan applies to this situation. In dimension three it yields, for every
\begin{equation*}
  \frac34<s<1,
\end{equation*}
the low-frequency bound
\begin{equation*}
  \sup_{t\ge0}\sum_{0<M\le1} M^{2s}\|P_Mu(t)\|_{L^2}^2<\infty.
\end{equation*}
For completeness, recall the algebraic content of this estimate. Pairing the forward and backward reduced-Duhamel representations for $P_Mu(t)$ and using the dispersive bound
\begin{equation*}
  \|e^{i\tau\Delta}f\|_{L^6(\R^3)}
  \lesssim |\tau|^{-1}\|f\|_{L^{6/5}(\R^3)},
  \qquad \tau\in\R\setminus\{0\},
\end{equation*}
produces a double-Duhamel inequality of the form
\begin{equation*}
  M^{2s}\|P_Mu(t)\|_{L^2}^2
  \lesssim c_M
  +\eta\sum_{K\le M}K^{2s}\|P_Ku\|_{L_t^\infty L_x^2}^2
  +C_\eta M^{2(1-s)}\|u\|_{L_t^\infty\dot H_x^1}^{10},
\end{equation*}
where $\eta>0$ is arbitrary, $C_\eta>0$ depends only on $\eta$, $c_M\ge0$, and
\begin{equation*}
  \sum_{0<M\le1}c_M<\infty.
\end{equation*}
The term $c_M$ is generated by the compact frequency tail, the middle term comes from interactions in which at least one factor has frequency $\le M$, and the last term is controlled by the conserved critical norm. Choosing $\eta$ sufficiently small and summing over dyadic $0<M\le1$ gives the stated low-frequency estimate. This is precisely the three-dimensional $L^p$ breach estimate in the range $s>3/4$; see \cite{KillipVisan2010,KillipVisan2013} for the full frequency-envelope derivation of the displayed recurrence.

The high-frequency part is easier and uses only the assumed $\dot H^1$ bound. Since $s<1$,
\begin{align*}
  \sup_{t\ge0}\sum_{M>1}M^{2s}\|P_Mu(t)\|_{L^2}^2
  &\le \sup_{t\ge0}\sum_{M>1}M^{2}\|P_Mu(t)\|_{L^2}^2  \\
  &\lesssim \sup_{t\ge0}\|\nabla u(t)\|_{L^2}^2
  <\infty.
\end{align*}
Combining the low- and high-frequency bounds gives
\begin{equation*}
  \sup_{t\ge0}\|u(t)\|_{\dot H^s}<\infty,
  \qquad \frac34<s<1.
\end{equation*}

Now fix $p\in(4,6)$ and set
\begin{equation*}
  s_p:=\frac32-\frac3p.
\end{equation*}
Then
\begin{equation*}
  \frac34<s_p<1,
  \qquad
  \frac1p=\frac12-\frac{s_p}{3}.
\end{equation*}
The homogeneous Sobolev embedding $\dot H^{s_p}(\R^3)\hookrightarrow L^p(\R^3)$ therefore gives
\begin{equation*}
  \|u(t)\|_{L^p}
  \lesssim_p \|u(t)\|_{\dot H^{s_p}}
\end{equation*}
for every $t\ge0$. Taking the supremum in time yields
\begin{equation*}
  \|u\|_{L_t^\infty L_x^p([0,\infty)\times\R^3)}<\infty.
\end{equation*}
This proves the proposition.
\end{proof}

\medskip
\noindent\textbf{Endpoint issue.} Proposition \ref{prop:Lp-breach} gives $L_t^\infty L_x^p$ for every $4<p<6$. The endpoint $p=4$ would correspond to $s=3/4$, where the frequency-envelope summation becomes critical. In low dimensions, the difficult point is to push further to endpoint $L^4$, finite mass, or negative regularity; the virial argument in this paper is designed to separate that issue from the localized virial computation itself.

\medskip
If an additional negative-regularity input is available, one may directly enter the finite-mass layer. Indeed, if a compact critical solution satisfies
\[
  \sup_{t\ge0}\|u(t)\|_{\dot H^{-s}}<\infty
  \qquad\text{for some }s>0,
  \qquad
  \sup_{t\ge0}\|u(t)\|_{\dot H^1}<\infty,
\]
then homogeneous Sobolev interpolation, with $\vartheta=s/(1+s)$ so that $0=(1-\vartheta)(-s)+\vartheta$, gives
\[
  \|u(t)\|_{L^2}
  \lesssim \|u(t)\|_{\dot H^{-s}}^{1-\vartheta}
           \|u(t)\|_{\dot H^1}^{\vartheta}.
\]
Thus $\sup_{t\ge0}\|u(t)\|_{L^2}<\infty$. This shows that the $q=2$ layer in the present paper can be triggered either by finite mass or by negative regularity; the three-dimensional difficulty is how to derive such negative regularity automatically from pure $\dot H^1$ critical compactness and the no-waste Duhamel structure.

\section{Fixed-center localized virial}\label{sec:virial}

Let $\mathbb N_0:=\{0,1,2,\ldots\}$. For a multi-index $\alpha=(\alpha_1,\alpha_2,\alpha_3)\in\mathbb N_0^3$, set
\begin{equation*}
  |\alpha|:=\alpha_1+\alpha_2+\alpha_3,
  \qquad
  \partial^\alpha:=\partial_1^{\alpha_1}\partial_2^{\alpha_2}\partial_3^{\alpha_3}.
\end{equation*}
For $j,k\in\{1,2,3\}$, write $\partial_j:=\partial/\partial x_j$, $\partial_{jk}:=\partial_j\partial_k$, and use the summation convention over repeated spatial indices. The symbol $\delta_{jk}$ denotes the Kronecker delta, and $\mathbf 1_E$ denotes the indicator function of a set $E$.

Choose a radial function $a\in C^\infty(\R^3)$ such that
\begin{align*}
  a(y)&=|y|^2, && |y|\le1,\\
  a(y)&=\text{constant}, && |y|\ge2,
\end{align*}
and, for every multi-index $\alpha$ with $|\alpha|\ge2$, there exists a constant $C_\alpha>0$ such that
\begin{equation*}
  |\partial^\alpha a(y)|\le C_\alpha.
\end{equation*}
For $R>0$, set
\begin{equation*}
  a_R(y):=R^2a(y/R).
\end{equation*}
Then
\begin{align}
  a_R(y)&=|y|^2, && |y|\le R,\notag\\
  \nabla a_R(y)&=0, && |y|\ge2R,\notag\\
  |\nabla a_R(y)|&\lesssim R,\label{eq:aR-grad-bound}\\
  |\partial^\alpha a_R(y)|&\lesssim R^{2-|\alpha|}, && |\alpha|\ge2.\label{eq:aR-derivative-bound}
\end{align}
Writing $\Delta^2:=\Delta\circ\Delta$, one has, in particular,
\begin{equation}\label{eq:Delta2-aR-bound}
  |\Delta^2a_R(y)|\lesssim R^{-2}\mathbf 1_{R\le|y|\le2R}.
\end{equation}
For later use, for $f\in\dot H^1(\R^3)$, $R>0$, and $z\in\R^3$, define the tail-error functional
\begin{equation*}
  \mathcal E_R[f;z]
  :=\int_{|x-z|\ge R}(|\nabla f(x)|^2+|f(x)|^6)\,dx
  +R^{-2}\int_{R\le|x-z|\le2R}|f(x)|^2\,dx.
\end{equation*}
A schematic normalized radial profile of this cutoff is shown in Figure~\ref{fig:virial-cutoff}. The exact formula is irrelevant; only the quadratic core, the compactly supported transition annulus, and the derivative bounds above enter the estimates.

\begin{figure}[t]
\centering
\begin{tikzpicture}[x=2.45cm,y=1.05cm, line cap=round, line join=round, >=Latex]
  \path[use as bounding box] (-.18,-.45) rectangle (3.38,3.55);
  \fill[gray!18] (1,0) rectangle (2,3.35);
  \draw[->,thick] (0,0) -- (3.18,0) node[below left=3pt] {$r=|x-z|$};
  \draw[->,thick] (0,0) -- (0,3.35) node[left=3pt] {cutoff profile};
  \draw[thin] (1,0) -- (1,3.28);
  \draw[thin] (2,0) -- (2,3.28);
  \node[below] at (1,0) {$R$};
  \node[below] at (2,0) {$2R$};
  \node[fill=white,inner sep=2pt] at (1.50,2.85) {$R\le r\le2R$};
  \draw[very thick]
    (0,0) .. controls (.18,.04) and (.65,.40) .. (1,1)
    .. controls (1.25,1.55) and (1.55,2.55) .. (2,3)
    -- (3.0,3);
  \draw[thick,dashed,gray!70]
    (0,0) .. controls (.35,.35) and (.70,.70) .. (1,1)
    .. controls (1.25,1.35) and (1.65,1.45) .. (2,0)
    -- (3.0,0);
  \node[fill=white,inner sep=2pt,anchor=west] at (.20,.52) {$a_R(r)=r^2$};
  \node[fill=white,inner sep=2pt,anchor=west] at (2.18,3.08) {constant};
  \node[fill=white,inner sep=2pt,anchor=west] at (2.18,2.45) {$a_R$};
  \node[fill=white,inner sep=2pt,anchor=west] at (2.18,1.40) {scaled $|\nabla a_R|$};
\end{tikzpicture}
\caption{A schematic radial virial cutoff. The region $r\le R$ gives the main term $8K(u)$, while the annulus $R\le r\le2R$ supports the truncation errors controlled by $\mathcal E_R$.}
\label{fig:virial-cutoff}
\end{figure}
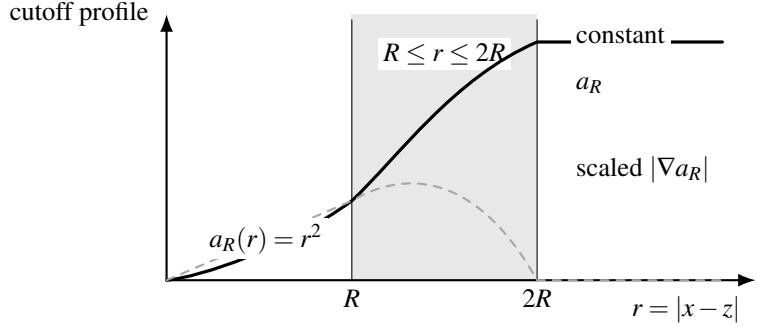

Once these cutoff functions have been fixed, with observation center $z=x(0)$ define the virial action
\begin{equation}\label{eq:virial-action}
  \Vc_R(t):=2\operatorname{Im}\int_{\R^3}
  \nabla a_R(x-z)\cdot\nabla u(t,x)\overline{u(t,x)}\,dx.
\end{equation}

The quantity \eqref{eq:virial-action} is meaningful for $\dot H^1$ solutions. Indeed, $\nabla a_R$ is supported in $B(z,2R)$ and $|\nabla a_R|\lesssim R$, while
\begin{equation*}
  \|u(t)\|_{L^2(B(z,2R))}\lesssim R\|\nabla u(t)\|_{L^2}
\end{equation*}
follows from Sobolev embedding. Hence
\begin{equation*}
  |\mathcal V_R(t)|
  \lesssim R\|\nabla u(t)\|_{L^2}\|u(t)\|_{L^2(B(z,2R))}<\infty.
\end{equation*}

To pass the virial computation from smooth solutions to general energy solutions, we first record the continuity of the truncated action under smoothing.

\begin{lemma}[Action continuity]\label{lem:virial-approximation}
Assume $f_n\to f$ in $\dot H^1(\R^3)$, and fix $R>0$ and $z\in\R^3$. Then
\[
  \int \nabla a_R(x-z)\cdot\nabla f_n\,\overline{f_n}\,dx
  \to
  \int \nabla a_R(x-z)\cdot\nabla f\,\overline f\,dx.
\]
Moreover, if smooth solutions $u_n$ converge to an energy solution $u$ on compact time intervals, then the localized virial identity passes to $u$.
\end{lemma}

\begin{proof}
By $\dot H^1\hookrightarrow L^6$ and local H\"older,
\[
  \|f_n-f\|_{L^2(B(z,2R))}
  \lesssim R\|\nabla(f_n-f)\|_{L^2}\to0.
\]
Thus
\[
\begin{aligned}
  &\left|\int \nabla a_R\cdot\nabla f_n\overline{f_n}
  -\int \nabla a_R\cdot\nabla f\overline f\right| \\
  &\quad\lesssim R\|\nabla(f_n-f)\|_{L^2}\|f_n\|_{L^2(B(z,2R))}
  +R\|\nabla f\|_{L^2}\|f_n-f\|_{L^2(B(z,2R))}\to0.
\end{aligned}
\]
Passing the identity also uses the convergence $f_n\to f$ in $L^6$ and the local continuity of the $L^2$ annular term; the proof is the same.
\end{proof}

With the smoothing approximation available, it is enough to compute first for smooth decaying solutions. This yields the localized virial identity used below.

\begin{lemma}[Localized virial]\label{lem:localized-virial}
If $u$ is sufficiently smooth and decaying, then
\begin{align}
  \frac{d}{dt}\Vc_R(t)
  =&\ 4\int_{\R^3}\partial_{jk}a_R(x-z)
  \operatorname{Re}\bigl(\partial_ju\overline{\partial_ku}\bigr)\,dx\notag\\
  &-\int_{\R^3}\Delta^2a_R(x-z)|u|^2\,dx
  -\frac43\int_{\R^3}\Delta a_R(x-z)|u|^6\,dx.
  \label{eq:localized-virial-identity}
\end{align}
For a general energy solution, \eqref{eq:localized-virial-identity} follows from Lemma \ref{lem:virial-approximation} and a standard smoothing approximation.
\end{lemma}

\begin{proof}
Substitute
\begin{equation*}
  \partial_tu=i\Delta u+i|u|^4u
\end{equation*}
into \eqref{eq:virial-action}. To make the sign convention explicit, write the time derivative as two parts:
\[
\begin{aligned}
  \frac{d}{dt}\mathcal V_R(t)
  &=2\operatorname{Im}\int \partial_j a_R\,\partial_j(\partial_tu)\,\overline u\,dx
    +2\operatorname{Im}\int \partial_j a_R\,\partial_j u\,\overline{\partial_tu}\,dx \\
  &=\mathcal L_R(t)+\mathcal N_R(t),
\end{aligned}
\]
where $\mathcal L_R$ collects the terms produced by $i\Delta u$, and $\mathcal N_R$ collects the terms produced by $i|u|^4u$. For the linear part, two integrations by parts and the symmetry $\partial_{jk}a_R=\partial_{kj}a_R$ give
\[
  \mathcal L_R(t)
  =4\int \partial_{jk}a_R\operatorname{Re}(\partial_j u\overline{\partial_k u})\,dx
   -\int \Delta^2a_R |u|^2\,dx.
\]
This step uses only that $a_R$ is real-valued, not the nonlinear structure. For the nonlinear term, use
\begin{equation*}
  \operatorname{Re}(|u|^4\overline u\,\nabla u)=\frac16\nabla(|u|^6),
\end{equation*}
which gives
\begin{equation*}
  -\frac43\int\Delta a_R|u|^6\,dx.
\end{equation*}
Combining the terms proves the identity.
\end{proof}

When $|x-z|\le R$,
\begin{equation}\label{eq:quadratic-derivatives}
  \partial_{jk}a_R=2\delta_{jk},
  \qquad \Delta a_R=6,
  \qquad \Delta^2a_R=0.
\end{equation}
Thus the main contribution is
\begin{equation*}
  8\int|\nabla u|^2\,dx-8\int|u|^6\,dx=8K(u).
\end{equation*}
We now write the difference between the localized virial derivative and the main term $8K$ as a unified tail error.

\begin{lemma}[Virial error]\label{lem:virial-error}
There exists a constant $C_a$, depending only on $a$, such that for every $f\in\dot H^1(\R^3)$, every $R>0$, and every $z\in\R^3$,
\begin{align*}
  &\left|
  4\int \partial_{jk}a_R(x-z)\operatorname{Re}(\partial_jf\overline{\partial_kf})\,dx
  -\int \Delta^2a_R(x-z)|f|^2\,dx \right.\\
  &\left.\hspace{2.2cm}
  -\frac43\int \Delta a_R(x-z)|f|^6\,dx
  -8K(f)
  \right|
  \le C_a\mathcal E_R[f;z],
\end{align*}
where
\begin{equation}\label{eq:error-functional}
  \mathcal E_R[f;z]
  :=\int_{|x-z|\ge R}(|\nabla f(x)|^2+|f(x)|^6)\,dx
  +R^{-2}\int_{R\le|x-z|\le2R}|f(x)|^2\,dx.
\end{equation}
\end{lemma}

\begin{proof}
Inside $|x-z|\le R$, use \eqref{eq:quadratic-derivatives}; the contribution is exactly $8K(f)$. All differences are supported in $\{|x-z|\ge R\}$. In this proof, define
\begin{align*}
  \Err_\nabla
  &:=4\int\bigl(\partial_{jk}a_R(x-z)-2\delta_{jk}\bigr)
       \operatorname{Re}(\partial_jf\overline{\partial_kf})\,dx,\\
  \Err_6
  &:=-\frac43\int\bigl(\Delta a_R(x-z)-6\bigr)|f|^6\,dx,\\
  \Err_2
  &:=-\int\Delta^2a_R(x-z)|f|^2\,dx.
\end{align*}
Then the expression on the left-hand side of the lemma is $|\Err_\nabla+\Err_6+\Err_2|$. By \eqref{eq:aR-derivative-bound}, the second-derivative term and the nonlinear term satisfy
\begin{equation*}
  |\Err_\nabla+\Err_6|
  \lesssim \int_{|x-z|\ge R}(|\nabla f|^2+|f|^6)\,dx.
\end{equation*}
By \eqref{eq:Delta2-aR-bound},
\begin{equation*}
  |\Err_2|
  \lesssim R^{-2}\int_{R\le|x-z|\le2R}|f|^2\,dx.
\end{equation*}
Combining the three terms gives the result.
\end{proof}

Define the total virial remainder $\operatorname{Err}_R(t;z)$ by
\begin{equation*}
  \frac{d}{dt}\mathcal V_R(t)=8K(u(t))+\operatorname{Err}_R(t;z).
\end{equation*}
By Lemma \ref{lem:virial-error}, the error satisfies
\begin{equation*}
  |\operatorname{Err}_R(t;z)|\le C_a\mathcal E_R[u(t);z].
\end{equation*}
This decomposition is the main axis of the rigidity proof: below-threshold coercivity gives $8K(u(t))\ge8\kappa_0$, while critical compactness and drift geometry allow one to impose $\mathcal E_{R_T}[u(t);z]\le\varepsilon_0\kappa_0$ for any prescribed $\varepsilon_0>0$. The next lemma converts this tail smallness into a uniform error control on each time window.

\begin{lemma}[Tail-error control]\label{lem:tail-error-control}
Assume that $u,N,x$ satisfy \eqref{eq:compactness-v} and \eqref{eq:N-bounded}. Given $\eps>0$, there exists $A=A(\eps)$ such that, for all $T\ge1$ and all $0\le t\le T$, if
\begin{equation*}
  R_T:=D_x(T)+A,
\end{equation*}
then
\begin{equation}\label{eq:ET-small}
  \mathcal E_{R_T}[u(t);z]\le\eps.
\end{equation}
\end{lemma}

\begin{proof}
By Lemma \ref{lem:compact-tail}, choose $A$ sufficiently large so that
\begin{equation}\label{eq:tail-small-A}
  \sup_{s\ge0}\int_{|x-x(s)|\ge A}
  (|\nabla u(s,x)|^2+|u(s,x)|^6)\,dx\le\eps.
\end{equation}
For $0\le t\le T$, \eqref{eq:geometry-tail} gives
\begin{equation}\label{eq:error-tail-inclusion}
  \{|x-z|\ge R_T\}\subset\{|x-x(t)|\ge A\}.
\end{equation}
Hence the gradient and $L^6$ tails in \eqref{eq:error-functional} are at most $\eps$.

It remains to control the $L^2$ annular term by H\"older's inequality. Let
\begin{equation*}
  \Omega_T:=\{x:R_T\le|x-z|\le2R_T\}.
\end{equation*}
By \eqref{eq:error-tail-inclusion}, $\Omega_T\subset\{|x-x(t)|\ge A\}$. Therefore
\begin{align*}
  R_T^{-2}\int_{\Omega_T}|u(t,x)|^2\,dx
  &\le R_T^{-2}|\Omega_T|^{2/3}
  \left(\int_{\Omega_T}|u(t,x)|^6\,dx\right)^{1/3}\\
  &\lesssim
  \left(\int_{|x-x(t)|\ge A}|u(t,x)|^6\,dx\right)^{1/3}.
\end{align*}
Replacing $\eps$ in \eqref{eq:tail-small-A} beforehand by $c\eps^3$, where $c>0$ is a sufficiently small absolute constant, gives \eqref{eq:ET-small}.
\end{proof}

We now combine the positive below-threshold gap with the tail-error control to obtain a linear lower bound for the fixed-center virial action.

\begin{proposition}[Virial derivative]\label{prop:virial-derivative-lower}
Let $u$ be a global strong solution satisfying
\begin{equation*}
  E(u_0)<E(W),
  \qquad \|\nabla u_0\|_{L^2}<\|\nabla W\|_{L^2},
\end{equation*}
and assume that there exist $N(t)$ and $x(t)$ such that \eqref{eq:compactness-v} and \eqref{eq:N-bounded} hold. If $u\not\equiv0$, then there exists $A_0>0$ such that, for every $T\ge1$, with
\begin{equation*}
  R_T:=D_x(T)+A_0,
\end{equation*}
one has
\begin{equation*}
  \frac{d}{dt}\Vc_{R_T}(t)\ge4\kappa_0,
  \qquad 0\le t\le T,
\end{equation*}
where $\kappa_0$ is given by \eqref{eq:K-kappa0}.
\end{proposition}

\begin{proof}
By Lemmas \ref{lem:localized-virial} and \ref{lem:virial-error},
\begin{equation}\label{eq:virial-derivative-error}
  \frac{d}{dt}\Vc_{R_T}(t)
  \ge8K(u(t))-C_a\mathcal E_{R_T}[u(t);z].
\end{equation}
By \eqref{eq:K-kappa0}, $8K(u(t))\ge8\kappa_0$. By Lemma \ref{lem:tail-error-control}, choose $A_0$ such that
\begin{equation*}
  C_a\mathcal E_{R_T}[u(t);z]\le4\kappa_0
\end{equation*}
for all $T\ge1$ and $0\le t\le T$. Substitution into \eqref{eq:virial-derivative-error} proves the claim.
\end{proof}

The derivative lower bound gives linear growth of the virial. To reach a contradiction, we also need to estimate the size of the same action at the endpoints of the time window.

\begin{proposition}[Action bound]\label{prop:virial-action-upper}
Assume $\sup_{t\ge0}\|\nabla u(t)\|_{L^2}<\infty$. Then, for every $R>0$ and $t\ge0$,
\begin{equation}\label{eq:virial-action-bound-general}
  |\Vc_R(t)|
  \lesssim_u R\,\|u(t)\|_{L^2(B(z,2R))}.
\end{equation}
If in addition $M_q(u)<\infty$, $2\le q\le6$, and if $0\le t\le T$ and $R=R_T=D_x(T)+A_0$, then
\begin{equation}\label{eq:virial-action-bound-RT}
  |\Vc_{R_T}(t)|
  \lesssim_{u,q,A_0} R_T^{1+\theta(q)}.
\end{equation}
\end{proposition}

\begin{proof}
By \eqref{eq:aR-grad-bound} and the Cauchy--Schwarz inequality,
\begin{align}
  |\Vc_R(t)|
  &\le2\int |\nabla a_R(x-z)|\,|\nabla u(t,x)|\,|u(t,x)|\,dx\notag\\
  &\lesssim R\|\nabla u(t)\|_{L^2}\|u(t)\|_{L^2(B(z,2R))}.
  \label{eq:virial-action-CS}
\end{align}
The assumption $\sup_{t\ge0}\|\nabla u(t)\|_{L^2}<\infty$ gives \eqref{eq:virial-action-bound-general}.
If $R=R_T$ and $0\le t\le T$, then by \eqref{eq:ball-inclusion},
\begin{equation*}
  \|u(t)\|_{L^2(B(z,2R_T))}
  \le\|u(t)\|_{L^2(B(x(t),3R_T))}
  \le\Phi(3R_T).
\end{equation*}
By Lemma \ref{lem:local-mass-Lq},
\begin{equation*}
  \Phi(3R_T)\lesssim_{u,q} R_T^{\theta(q)}.
\end{equation*}
Substituting this into \eqref{eq:virial-action-CS} gives \eqref{eq:virial-action-bound-RT}.
\end{proof}

Proposition \ref{prop:virial-action-upper} immediately gives three common action bounds. If $R_T=D_x(T)+A_0$, then, for $0\le t\le T$,
\[
\begin{aligned}
  q=6:\quad& |\mathcal V_{R_T}(t)|\lesssim R_T^2,\\
  q=4:\quad& |\mathcal V_{R_T}(t)|\lesssim R_T^{7/4}
        \quad\text{if }u\in L_t^\infty L_x^4,\\
  q=2:\quad& |\mathcal V_{R_T}(t)|\lesssim R_T
        \quad\text{if }u\in L_t^\infty L_x^2.
\end{aligned}
\]
Thus the main proof below only compares the linear lower bound $\mathcal V_{R_T}(T)-\mathcal V_{R_T}(0)\gtrsim T$ with the estimate above to obtain the corresponding drift contradiction.

\section{Proofs of the main results}\label{sec:main-proofs}

This section closes the three estimates obtained above. Theorem \ref{thm:rigidity-general} is proved by contradiction, excluding nonzero critical elements, and Theorem \ref{thm:scattering-conditional} then turns this exclusion into conditional scattering.

\begin{proof}[Proof of Theorem \ref{thm:rigidity-general}]
Assume, for contradiction, that $U\not\equiv0$. To avoid repeated notation, write $u$ for $U$ in the proof. By below-threshold coercivity and Lemma \ref{lem:positive-K}, there exists $\kappa_0>0$ such that
\begin{equation*}
  K(u(t))\ge\kappa_0,
  \qquad t\ge0.
\end{equation*}
By Proposition \ref{prop:virial-derivative-lower}, there exists $A_0>0$ such that, for each $T\ge1$, if
\begin{equation*}
  R_T:=D_x(T)+A_0,
\end{equation*}
then the fixed-center virial action satisfies
\begin{equation*}
  \frac{d}{dt}\Vc_{R_T}(t)\ge4\kappa_0,
  \qquad 0\le t\le T.
\end{equation*}
Integrating gives
\begin{equation*}
  \Vc_{R_T}(T)-\Vc_{R_T}(0)\ge4\kappa_0T.
\end{equation*}
On the other hand, by Proposition \ref{prop:virial-action-upper},
\begin{equation*}
  |\Vc_{R_T}(T)|+|\Vc_{R_T}(0)|
  \lesssim_{u,q,A_0} R_T^{1+\theta(q)}.
\end{equation*}
Consequently,
\begin{equation}\label{eq:proof-T-less-RT}
  T\lesssim_{u,q,A_0} R_T^{1+\theta(q)}.
\end{equation}
By the hypothesis \eqref{eq:rigidity-drift-general}, there exists a sequence $T_n\to\infty$ such that
\begin{equation*}
  \frac{(D_x(T_n)+1)^{1+\theta(q)}}{T_n}\to0.
\end{equation*}
Since $A_0$ is fixed,
\begin{equation*}
  \frac{(D_x(T_n)+A_0)^{1+\theta(q)}}{T_n}\to0.
\end{equation*}
Taking $T=T_n$ in \eqref{eq:proof-T-less-RT}, we get
\begin{equation*}
  1\lesssim_{u,q,A_0}
  \frac{R_{T_n}^{1+\theta(q)}}{T_n}\to0,
\end{equation*}
a contradiction. Therefore $U\equiv0$.
\end{proof}

\begin{proof}[Proof of Theorem \ref{thm:scattering-conditional}]
Assume, for contradiction, that there exists $u_0\in\mathfrak A$ such that the corresponding forward global strong solution $u$ does not scatter. By the standard energy-critical local theory, this means
\begin{equation*}
  S_{[0,\infty)}(u)=\infty.
\end{equation*}
By the low-speed soliton-like critical-element reduction property assumed in the theorem, there exists a nonzero global strong solution $U$, and $U$ satisfies all assumptions of Theorem \ref{thm:rigidity-general}. Hence Theorem \ref{thm:rigidity-general} gives
\begin{equation*}
  U\equiv0,
\end{equation*}
which contradicts the nonzero critical element obtained from the reduction. Therefore, for every $u_0\in\mathfrak A$,
\begin{equation*}
  S_{[0,\infty)}(u)<\infty.
\end{equation*}
Finally, by Proposition \ref{prop:finite-S-scattering}, the finite norm $S_{[0,\infty)}(u)$ implies that there exists $u_+\in\dot H^1(\R^3)$ such that
\begin{equation*}
  \lim_{t\to\infty}\|u(t)-e^{it\Delta}u_+\|_{\dot H^1}=0.
\end{equation*}
This proves the theorem.
\end{proof}

\medskip
\noindent\textbf{Core estimates.} Fix a sufficiently small absolute constant $c>0$. The proof of Theorem \ref{thm:rigidity-general} uses only the following three classes of inequalities:
\begin{align}
  \text{coercivity:}\quad &K(u(t))\ge\kappa_0,
  \label{eq:ABC-A}\\
  \text{small error:}\quad &\mathcal E_{D_x(T)+A_0}[u(t);x(0)]\le c\kappa_0,
  \qquad 0\le t\le T,
  \label{eq:ABC-B}\\
  \text{action upper bound:}\quad &|\mathcal V_{D_x(T)+A_0}(t)|
  \lesssim (D_x(T)+A_0)^{1+\theta(q)}.
  \label{eq:ABC-C}
\end{align}
Here \eqref{eq:ABC-A} comes from the sharp Sobolev threshold, \eqref{eq:ABC-B} from critical compactness and drift geometry, and \eqref{eq:ABC-C} from local mass growth. This decomposition shows that the additional hypothesis $L_t^\infty L_x^q$ enters only in the action upper bound and does not alter the below-threshold coercivity or the tail-error control.

\medskip
\noindent\textbf{Unconditional direction.} The version without additional integrability proved here is the $q=6$ exclusion of subdiffusive drift. The standard $L^p$ breach in critical-element theory pushes the drift exponent to every $\beta<4/7$; if one also has the endpoint $L_t^\infty L_x^4$ control, then the same proof gives the endpoint $4/7$ and the sequential condition. If one could further prove that three-dimensional non-radial soliton-like critical elements have finite mass, then the same virial mechanism would immediately exclude all sublinear drift. Thus, from the viewpoint of the present method, the main obstruction to the full sublinear version is not the localized virial error, but the negative-regularity or finite-mass problem at the three-dimensional low-dimensional endpoint. Complete unconditional scattering also requires an additional proof that every nonscattering counterexample can be reduced to a low-speed soliton-like channel of the type excluded here.

\begin{acknowledgements}
The author would like to thank the supporting agencies for their continued support.
\end{acknowledgements}

\paragraph{Funding}
The author was supported by the Project ``Research on Nonlinear Partial Differential Equations'' (No. 2024KYCXTD018), the Special Projects in Key Areas of Guangdong Province (No. ZDZX1088), and the Fund of Guangzhou Municipal Science and Technology (No. 202102080428).

\paragraph{Data Availability Statement}
Data sharing is not applicable to this article, as no datasets were generated or analyzed during the current study.

\section*{Declarations}
\paragraph{Conflict of interest}
The author declares that he has no conflict of interest.

\paragraph{Publisher's Note}
Springer Nature remains neutral with regard to jurisdictional claims in published maps and institutional affiliations.

\begingroup
\sloppy
\hbadness=10000

\endgroup


\begin{thebibliography}{99}

\bibitem{Aubin1976}
Aubin, T.: Probl\`emes isop\'erim\'etriques et espaces de Sobolev. J. Differential Geometry 11, 573--598 (1976)

\bibitem{Talenti1976}
Talenti, G.: Best constant in Sobolev inequality. Ann. Mat. Pura Appl. 110, 353--372 (1976)

\bibitem{Weinstein1983}
Weinstein, M.I.: Nonlinear Schr\"odinger equations and sharp interpolation estimates. Commun. Math. Phys. 87, 567--576 (1983)

\bibitem{GinibreVelo1979I}
Ginibre, J., Velo, G.: On a class of nonlinear Schr\"odinger equations. I. The Cauchy problem, general case. J. Funct. Anal. 32, 1--32 (1979)

\bibitem{GinibreVelo1979II}
Ginibre, J., Velo, G.: On a class of nonlinear Schr\"odinger equations. II. Scattering theory, general case. J. Funct. Anal. 32, 33--71 (1979)

\bibitem{GinibreVelo1985}
Ginibre, J., Velo, G.: The global Cauchy problem for the nonlinear Schr\"odinger equation revisited. Ann. Inst. H. Poincar\'e Anal. Non Lin\'eaire 2(4), 309--327 (1985)

\bibitem{Kato1987}
Kato, T.: On nonlinear Schr\"odinger equations. Ann. Inst. H. Poincar\'e Phys. Th\'eor. 46, 113--129 (1987)

\bibitem{CazenaveWeissler1990}
Cazenave, T., Weissler, F.B.: The Cauchy problem for the critical nonlinear Schr\"odinger equation in $H^s$. Nonlinear Anal. 14, 807--836 (1990)

\bibitem{KeelTao1998}
Keel, M., Tao, T.: Endpoint Strichartz estimates. Am. J. Math. 120, 955--980 (1998)

\bibitem{Cazenave2003}
Cazenave, T.: Semilinear Schr\"odinger Equations. Courant Lecture Notes in Mathematics, vol. 10. American Mathematical Society, Providence (2003)

\bibitem{Glassey1977}
Glassey, R.T.: On the blowing up of solutions to the Cauchy problem for nonlinear Schr\"odinger equations. J. Math. Phys. 18(9), 1794--1797 (1977)

\bibitem{OgawaTsutsumi1991}
Ogawa, T., Tsutsumi, Y.: Blow-up of $H^1$ solution for the nonlinear Schr\"odinger equation. J. Differential Equations 92, 317--330 (1991)

\bibitem{Bourgain1999}
Bourgain, J.: Global well-posedness of defocusing critical nonlinear Schr\"odinger equation in the radial case. J. Am. Math. Soc. 12(1), 145--171 (1999)

\bibitem{Tao2005}
Tao, T.: Global well-posedness and scattering for the higher-dimensional energy-critical nonlinear Schr\"odinger equation for radial data. N. Y. J. Math. 11, 57--80 (2005)

\bibitem{TaoVisan2005}
Tao, T., Visan, M.: Stability of energy-critical nonlinear Schr\"odinger equations in high dimensions. Electron. J. Differential Equations 2005(118), 1--28 (2005)

\bibitem{Visan2007}
Visan, M.: The defocusing energy-critical nonlinear Schr\"odinger equation in higher dimensions. Duke Math. J. 138, 281--374 (2007)

\bibitem{CKSTT2008}
Colliander, J., Keel, M., Staffilani, G., Takaoka, H., Tao, T.: Global well-posedness and scattering for the energy-critical nonlinear Schr\"odinger equation in $\mathbb R^3$. Ann. of Math. 167, 767--865 (2008)

\bibitem{BahouriGerard1999}
Bahouri, H., G\'erard, P.: High frequency approximation of solutions to critical nonlinear wave equations. Am. J. Math. 121, 131--175 (1999)

\bibitem{Keraani2001}
Keraani, S.: On the defect of compactness for the Strichartz estimates of the Schr\"odinger equations. J. Differential Equations 175(2), 353--392 (2001)

\bibitem{KenigMerle2006}
Kenig, C.E., Merle, F.: Global well-posedness, scattering and blow-up for the energy-critical focusing non-linear Schr\"odinger equation in the radial case. Invent. Math. 166, 645--675 (2006)

\bibitem{DuyckaertsMerle2009}
Duyckaerts, T., Merle, F.: Dynamic of threshold solutions for energy-critical NLS. Geom. Funct. Anal. 18, 1787--1840 (2009)

\bibitem{HolmerRoudenko2008}
Holmer, J., Roudenko, S.: A sharp condition for scattering of the radial 3D cubic nonlinear Schr\"odinger equation. Commun. Math. Phys. 282, 435--467 (2008)

\bibitem{DuyckaertsHolmerRoudenko2008}
Duyckaerts, T., Holmer, J., Roudenko, S.: Scattering for the non-radial 3D cubic nonlinear Schr\"odinger equation. Math. Res. Lett. 15(6), 1233--1250 (2008)

\bibitem{KillipVisan2010}
Killip, R., Visan, M.: The focusing energy-critical nonlinear Schr\"odinger equation in dimensions five and higher. Am. J. Math. 132(2), 361--424 (2010)

\bibitem{KillipVisan2013}
Killip, R., Visan, M.: Nonlinear Schr\"odinger equations at critical regularity. In: Evolution Equations, Clay Mathematics Proceedings, vol. 17, pp. 325--437. American Mathematical Society, Providence (2013)

\bibitem{Dodson2019}
Dodson, B.: Global well-posedness and scattering for the focusing, cubic Schr\"odinger equation in dimension $d=4$. Ann. Sci. \`Ec. Norm. Sup\'er. 52(1), 139--180 (2019)

\end{thebibliography}
\end{document}